\documentclass[a4paper,12pt]{article}
\usepackage[T1]{fontenc}
\usepackage[utf8]{inputenc}
\usepackage{amsmath}
\usepackage{amssymb}
\usepackage{amsthm}
\usepackage{amsfonts}
\usepackage{graphicx}
\usepackage{hyperref}
\usepackage{url}
\usepackage{array}
\usepackage{fancyhdr}
\usepackage{float}
\usepackage{chngcntr}
\usepackage{empheq}
\usepackage{hyperref}
\usepackage[nameinlink,capitalise,noabbrev]{cleveref}
\usepackage{mathrsfs}

\usepackage{tikz,tikz-3dplot, subcaption, tikz-cd}
\usetikzlibrary{cd,arrows,chains,matrix,positioning,scopes,decorations.markings}
\makeatletter
\tikzset{join/.code=\tikzset{after node path={%
\ifx\tikzchainprevious\pgfutil@empty\else(\tikzchainprevious)%
edge[every join]#1(\tikzchaincurrent)\fi}}}
\tikzset{>=stealth',every on chain/.append style={join},
         every join/.style={->}}
\tikzcdset{every label/.append style = {font = \normalsize}, every matrix/.append style={nodes={font=\large}}}
\tikzstyle{arrow} = [->,>=stealth]

\newcommand{\PP}{\mathbb{P}}

\newcommand{\Fq}{\mathbb{F}_q}

\newcommand{\slq}{SL_2\big(\Fq\big)}

\newcommand{\FF}{\mathbb{F}}
\newcommand{\fplane}{\PP^2\big(\FF\big)}

\newcommand{\m}{\mathfrak{m}}

\newcommand{\ord}{\mathrm{ord}}
\newcommand{\abcd}{\begin{pmatrix} \alpha & \beta \\ \gamma & \delta \end{pmatrix}}

\newcommand{\overbar}[1]{\mkern 1.5mu\overline{\mkern-1.5mu#1\mkern-1.5mu}\mkern 1.5mu}
\def\bign#1{\mathclose{\hbox{$\left#1\vbox to8.5\p@{}\right.\n@space$}}\mathopen{}}

\theoremstyle{plain}
\newtheorem{thm}{Theorem}[section]

\newtheorem{lemma}[thm]{Lemma}

\newtheorem{prop}[thm]{Proposition}

\newtheorem{corollary}[thm]{Corollary}

\theoremstyle{definition}

\theoremstyle{remark}
\newtheorem{rmk}[thm]{Remark}

\theoremstyle{definition}

\title{The canonical representation of the \\Drinfeld curve}
\author{{\sc Lucas Laurent} and {\sc Bernhard Köck}}
\date{}

\begin{document}

\maketitle

\begin{abstract}
\noindent
If $C$ is a smooth projective curve over an algebraically closed field $\mathbb{F}$ and $G$ is a group of automorphisms of $C$, the \textit{canonical representation of $C$} is given by the induced $\FF$-linear action of $G$ on the vector space $H^0\big(C,\Omega_C\big)$ of holomorphic differentials on $C$. Computing it is still an open problem in general when the cover $C \rightarrow C/G$ is wildly ramified. In this paper, we fix a prime power $q$, we consider the Drinfeld curve, i.e., the curve $C$ given by the equation ${XY^q-X^qY-Z^{q+1}=0}$ over $\mathbb{F}=\overbar{\mathbb{F}_q}$ together with its standard action by ${G=SL_2\big(\mathbb{F}_q\big)}$, and decompose $H^0\big(C,\Omega_C\big)$ as a direct sum of indecomposable representations of~$G$, thus solving the aforementioned problem in this case.
\end{abstract}

\section{Introduction}
The action of a finite group $G$ on a smooth projective curve~$C$ over an algebraically closed field $\FF$ induces an $\FF$-linear action of $G$ on the vector space of holomorphic differentials on $C$, denoted by $H^0\big(C,\Omega_C\big)$. This space hence becomes an $\FF[G]$-module whose dimension is equal to the genus $g(C)$ of the curve $C$, a fundamental invariant of smooth projective curves. Thus, it is natural to try and compute it, i.e., to find the decomposition of $H^0\big(C,\Omega_C\big)$ as a direct sum of indecomposable $\FF[G]$-modules.

This problem has been introduced by Hecke in 1928 \cite{hecke}, and it has been solved in 1934 by Chevalley and Weil \cite{chevalley} in the case the characteristic of~$\FF$ does not divide the order of $G$. If in contrast the characteristic~$p$ of~$\FF$ does divide the order of $G$, then $H^0\big(C,\Omega_C\big)$ is a modular representation, and this problem becomes notoriously hard, one reason being that already for fairly small groups listing the indecomposable representations is a formidable problem.

However, some progress has been made. In 1986, Nakajima \cite{nakajima} and Kani \cite{kani} independently computed this representation for an arbitrary finite group of automorphisms $G$ of $C$, in the case where the cover $C \rightarrow C/G$ is at most tamely ramified. Moreover, most of the theorems leading to Kani’s and Nakajima’s result have been generalised in \cite{kock} to the so-called weakly ramified case, the simplest but, in a sense, most frequent form of wild ramification.

On the other hand, restricting the type of the group $G$ rather than the ramification of $C \rightarrow C/G$ has lead to some progress as well. Indeed, Valentini and Madan \cite{valentini} solved the problem for $G$ a cyclic $p$-group. Karanikolopoulos and Kontogeorgis \cite{karanikolopoulos} extended this result to an arbitrary cyclic group~$G$. Moreover, Bleher, Chinburg, and Kontogeorgis \cite{bleher} gave an algorithm to solve the problem in the case when $G$ admits a non-trivial cyclic Sylow $p$-subgroup, which is a necessary and sufficient condition for $G$ to admit finitely many isomorphism classes of indecomposable representations \cite{higman}. Further results in this context have been obtained by Rzedowski-Calder\'{o}n, Villa-Salvador and Madan in \cite{madanvilla1} and \cite{madanvilla2}, by Marques and Ward in \cite{ward}, by Garnek in \cite{garnek} and by Bleher and Camacho in~\cite{blehercamacho}.

The object of this paper is to find the decomposition of $H^0\big(C,\Omega_C\big)$ as a direct sum of indecomposable $\FF[G]$-modules when $G=\slq$, $q=p^r$ for some $r \geq 1$, $\FF=\overbar{\Fq}$, and $C$ is the Drinfeld curve together with the action of $G$ as presented in \cref{section context}. The precise result is stated in \cref{decomposition thm}.

Note that because the Sylow $p$-subgroups of $G$ are isomorphic to the additive group of $\FF_q$, we can apply the algorithm from \cite{bleher} if, and only if, $r=1$. In that paper, the authors give a formula for the decomposition of $\mathrm{Res}^G_H H^0\big(C,\Omega_C\big)$ where $H = P \rtimes C'$ is a $p$-hypo-elementary subgroup of $G$, i.e., $P$ is a cyclic $p$-subgroup and $C'$ is a cyclic $p'$-subgroup. Conlon's induction theorem \cite[Corollary 80.51]{curtisreiner} implies that the $\FF[G]$-module structure of $H^0\big(C,\Omega_C\big)$ is completely determined by this. This procedure has led to a first proof of \cref{decomposition thm} in the case $r=1$, see \cite[Chapter~2]{laurent}.

In order to prove \cref{decomposition thm} for an arbitrary $r \geq 1$, we however use the following very different approach. In \cref{section basis}, we first compute an explicit basis of the vector space $H^0\big(C,\Omega_C\big)$ over $\FF$. Then from the action of $G$ on this basis, we find a decomposition of the $\FF[G]$-module $H^0\big(C,\Omega_C\big)$ in \cref{section decomposition}, which has been inspired by the case $r=1$. From   \cite[Corollary~10.1.5]{bonnafe} we obtain that each summand in this decomposition is indecomposable.

{\bf Acknowledgements}. The second author thanks Adriano Marmora for early, extensive and insightful discussions about the general theme of this paper where the specific problem underlying this paper has arisen. The authors also thank Peter Kropholler for his continued interest in our research. Finally, the authors thank the referees for carefully reading the paper and for their constructive comments, particularly for pointing out that the just mentioned indecomposability can be easily derived from the literature.

\section{Preliminaries about the Drinfeld curve and statement of main result}\label{section context}
In this section, we introduce some notations and facts about the Drinfeld curve and state our main result.

Throughout, we fix an algebraically closed field $\mathbb{F}$ of characteristic $p >0$ and a power $q=p^r$ of $p$. The {\it Drinfeld curve $C$} is then defined as the zero locus in $\fplane$ of the homogeneous polynomial
\[XY^q-X^qY-Z^{q+1},\]
which is the compactification of the affine curve $C_Z$ defined by $xy^q-x^qy-1$. This curve has $q+1$ points at infinity defined by $XY^q-X^qY=0$, which are the points $[x:y:0]$ where $x,y \in \Fq$. As seen in \cite[Propositions 2.1.1 and 2.4.1]{bonnafe}, the projective curve $C$ is smooth and irreducible. Usually, we will simply write
\[M:= H^0\big(C,\Omega_C\big)\]
for the vector space over $\FF$ of holomorphic differentials on $C$. By definition, $M$ has dimension $g(C)$, the genus of $C$, which by the genus-degree formula is $g(C)=\frac{q(q-1)}{2}$ \cite[Subsection 2.5.1]{bonnafe}. We denote the function field of $C$ by $\FF(C)$.

We view $\FF_q$ as the unique subfield of~$\FF$ containing $q$ elements.  For $g=\abcd \in G :=\slq$ and $[x:y:z] \in C$, we then define
\[g \cdot [x:y:z] := [\alpha x+\beta y:\gamma x+\delta y:z].\]
It is straightforward to verify that the map
\[
G \times C  \longrightarrow C, \quad
(g,[x:y:z]) \longmapsto g \cdot [x:y:z]
\]
defines a left action of $G$ on the algebraic curve $C$. Pulling back differentials finally defines a right action of $G$ on $M$ which is called the {\it canonical representation of $C$} and which is the object of study of this paper. Our main result, see \cref{decomposition thm} below, computes this modular representation.

We view the elements of $\FF^2$ as row vectors and and consider the standard right action of $G=\slq$ on $\FF^2$. For $k \ge 0$, let $V^k$ denote the $k^\mathrm{th}$ symmetric power $\mathrm{Sym}^k(\FF^2)$ of $\FF^2$ together with its induced action of $G$.

\begin{thm}\label{decomposition thm}
We have the following decomposition of $H^0(C, \Omega_C)$ into indecomposable $\FF[G]$-modules:
\[H^0\big(C,\Omega_C\big) \cong \bigoplus\limits_{k=0}^{q-2} V^k\]
In particular, the $\FF[G]$-module $H^0(C, \Omega_C)$ is semi-simple if and only if $q=p$.
\end{thm}

\begin{proof}
Establishing this isomorphism of $\FF[G]$-modules, see  \cref{decomposition}, is the object of the remaining two sections of this paper.

If $0 \le k \le q-1$, then, by \cref{remark geometric interpretation} below and \cite[Corollary 10.1.5]{bonnafe}, the intersection of all non-zero $\FF[G]$-submodules of $V^k$ is non-zero; in particular, $V^k$ is indecomposable.

Furthermore, if $q=p$, the $\FF[G]$-modules $V^0, \ldots, V^{p-1}$ are simple by \cite[p.\ 15--16]{alperin}. (In fact, they form a full set of non-isomorphic simple $\FF[G]$-modules). On the other hand, if $q >p$, then $V^p$ is not simple; this follows for example from \cite[Theorem 10.1.8]{bonnafe}.
\end{proof}

\section{A basis of the space of global holomorphic differentials}\label{section basis}
The goal of this section is to exhibit an explicit basis of the vector space $M$.

First, we identify a local parameter at each point of $C$. These will be used to verify that certain differentials on $C$ are holomorphic.

\begin{lemma}\label{local parameters}
Let $P=[x_0:y_0:z_0]$ be a point of $C$.
\begin{itemize}
\item If $z_0 \neq 0$, then a local parameter at $P$ is $\frac{X}{Z}-\frac{x_0}{z_0}$.
\item If $z_0=0$ and $x_0 \neq 0$, then a local parameter at $P$ is $\frac{Z}{X}$.
\item If $z_0=0$ and $y_0 \neq 0$, then a local parameter at $P$ is $\frac{Z}{Y}$.
\end{itemize}
\end{lemma}

\begin{proof}
Let us begin with a point $[x_0:y_0:z_0] \in C_Z$. Let ${x=\frac{X}{Z}}$ and $y=\frac{Y}{Z}$ denote the two standard coordinates on the affine chart $C_Z$. Then $P$ becomes $(\tilde{x_0},\tilde{y_0}):=\left(\frac{x_0}{z_0},\frac{y_0}{z_0}\right)$. By Nakayama's Lemma, as the local ring $\mathcal{O}_{C_Z,(\tilde{x_0},\tilde{y_0})}$ is a discrete valuation ring \cite[Theorems I.5.1 and I.6.2A]{hartshorne}, one of the two generators $x-\tilde{x_0}$ and $y-\tilde{y_0}$ of the maximal ideal $\m$ of $\mathcal{O}_{C_Z,(\tilde{x_0},\tilde{y_0})}$ will be a local parameter at $(\tilde{x_0},\tilde{y_0})$. Using the equations ${\tilde{x_0}\tilde{y_0}^q-\tilde{x_0}^q\tilde{y_0}=1=xy^q-x^qy}$, we obtain that
\begin{align*}
(x-\tilde{x_0})\tilde{y_0}^q-\tilde{x_0}^q(y-\tilde{y_0})& = -(\tilde{x_0}\tilde{y_0}^q-\tilde{x_0}^q\tilde{y_0})+x\tilde{y_0}^q-\tilde{x_0}^qy \\
& = -(xy^q-x^qy)+x\tilde{y_0}^q-\tilde{x_0}^qy\\
&=y(x-\tilde{x_0})^q-x(y-\tilde{y_0})^q \in \m^2.
\end{align*}
Therefore, because both $\tilde{x_0}$ and $\tilde{y_0}$ are non-zero, $x-\tilde{x_0}$ and $y-\tilde{y_0}$ differ mod~$\mathfrak{m}^2$ by a unit, hence each of them generates $\m$. Without loss of generality, we choose $x-\tilde{x_0}=\frac{X}{Z}-\frac{x_0}{z_0}$.\\
Now, suppose that $z_0=0$. Then either $x_0 \neq 0$ or $y_0 \neq 0$. If $x_0 \neq 0$, we use the coordinates $s=\frac{Y}{X}$ and $t=\frac{Z}{X}$ and write $P$ as $(s_0,t_0):=\left(\frac{y_0}{x_0},0\right)$. Since $s^q-s-t^{q+1}=0$ and $s_0 \in \Fq$, we obtain
\[s-s_0=s^q-t^{q+1}-s_0=s^q-t^{q+1}-s_0^q=(s-s_0)^q-(t-t_0)^{q+1},\]
hence $s-s_0 \in \m^q \subseteq \m^2$, which cannot be true for a local parameter. Thus, arguing as above, $t-t_0=\frac{Z}{X}$ is a local parameter at $P$.\\
Finally, if $y_0 \neq 0$ we use the coordinates $v=\frac{X}{Y}$ and $w=\frac{Z}{Y}$ and $P$ becomes $(v_0,w_0):=\left(\frac{x_0}{y_0},0\right)$. Then, again because $v-v^q-w^{q+1}=0$ and $v_0 \in \Fq$, we similarly obtain that $v-v_0 \in \m^q$ and that $w -w_0=\frac{Z}{Y}$ is a local parameter.
\end{proof}

The next proposition exhibits a basis of the vector space $M$. As in the proof above, let $x=\frac{X}{Z}$ and $y=\frac{Y}{Z}$. Recall \cite[Theorem II.8.6A]{hartshorne} that the vector space $\Omega_{\FF(C)/\FF}$ of meromorphic differentials on $C$ is of dimension 1 over $\FF(C)$, for example with basis $dx$.

\begin{prop}\label{basis}
For $0 \leq i,j \leq q-2$ such that $i+j \leq q-2$, define
\[\omega_{i,j} :=\frac{x^iy^j}{x^q}dx \in \Omega_{\FF(C)/\FF}. \]
Then each $\omega_{i,j}$ is holomorphic on the whole of $C$, i.e., $\omega_{i,j} \in M$, and the set $\{\omega_{i,j} \ | \ 0 \leq i,j \leq q-2, \ i+j \leq q-2 \}$ is a basis of $M$ over $\FF$.
\end{prop}

\begin{proof}
We will first show that each $\omega_{i,j}$ is in $M$. It then suffices to show that they are linearly independent over $\FF$ because the dimension of $M$ is $g(C)= \frac{(q-1)q}{2}$ which is easily seen to be equal to the number of $\omega_{i,j}$.\\
We know, if $t$ is a local parameter at any $P \in C$, then the meromorphic differential $f dt$ is regular at $P$ if, and only if, the rational function $f$ is regular at $P$ \cite[Section IV.2]{hartshorne}. Let now $0 \leq i,j \leq q-2$ such that $i+j \leq q-2$, and let $P=[x_0:y_0:z_0]\in C$. Suppose first that $z_0 \neq 0$. Then, because $x-\frac{x_0}{z_0}$ is a local parameter at $P$ by \cref{local parameters} and because $dx=d(x-\frac{x_0}{z_0})$, we simply need to verify that $\frac{x^iy^j}{x^q}$ is a regular function at $P$. This is the case because the value $\frac{x_0}{z_0}$ of $x$ at $P$ is non-zero because $\frac{x_0}{z_0}\left(\frac{y_0}{z_0}\right)^q-\left(\frac{x_0}{z_0}\right)^q\frac{y_0}{z_0}=1$. Therefore, $\omega_{i,j}$ is holomorphic on $C_Z$. Now, suppose that $z_0=0$ and $x_0 \neq 0$, and let $s=\frac{Y}{X}$, $t=\frac{Z}{X}$ as in the proof of \cref{local parameters}. Then $x=\frac{1}{t}$ and $y=\frac{s}{t}$, so the following holds in $\Omega_{\FF(C)/\FF}$.
\[\omega_{i,j}=\frac{\left(\frac{1}{t}\right)^i\left(\frac{s}{t}\right)^j}{\left(\frac{1}{t}\right)^q}d\left(\frac{1}{t}\right)=t^{q-i-j}s^j\left(\frac{-1}{t^2}\right)dt=-t^{q-2-(i+j)}s^j dt\]
Because $t$ is a local parameter at $P$ by \cref{local parameters}, we just need to check that $-t^{q-2-(i+j)}s^j$ is regular at $P$, which is the case because $x_0 \neq 0$ and $i+j \leq q-2$. Finally, suppose that $z_0=0$ and $y_0 \neq 0$, and let $v=\frac{X}{Y},$ $w=\frac{Z}{Y}$ as in \cref{local parameters}. By observing that $x=\frac{v}{w}$ and $y=\frac{1}{w}$, combined with the fact that $\frac{1}{x^q} dx=\frac{1}{y^q} dy \text{ in } \Omega_{\FF(C)/\FF}$ (obtained by differentiating $xy^q-x^qy-1=0$), the following holds.
\[\omega_{i,j}=\frac{\left(\frac{v}{w}\right)^i\left(\frac{1}{w}\right)^j}{\left(\frac{v}{w}\right)^q} v^q d\left(\frac{1}{w}\right)=w^{q-i-j}v^i \left(\frac{-1}{w^2}\right) dw=-w^{q-2-(i+j)}v^{i} dw\]
As above, because $w$ is a local parameter at $P$ by \cref{local parameters} and because $-w^{q-2-(i+j)}v^{i}$ is regular at $P$, we have that $\omega_{i,j}$ is regular at $P$. We proved that each $\omega_{i,j}$ is an element of $M$.\\
Let us now show that these differentials are linearly independent over $\FF$. Let $\lambda_{i,j} \in \FF$ such that
$\sum_{i,j} \lambda_{i,j} \omega_{i,j}=0$ in $M$. Then
\[\left(\sum\limits_{i,j} \lambda_{i,j} x^iy^j\right)\frac{1}{x^q} dx=\sum\limits_{i,j} \lambda_{i,j} x^iy^j \frac{1}{x^q} dx=0 \quad \textrm{in} \quad \Omega_{\FF(C)/\FF}\]
and hence $f(x,y):=\sum\limits_{i,j} \lambda_{i,j} x^iy^j=0$ in $\FF(C)$. Therefore, the affine variety~$C_Z$ is a subset of the zero locus of $f(x,y)$ in $\mathbb{A}^2(\FF)$, which implies by Hilbert's Nullstellensatz that
\[f(x,y) \in \text{Rad}\big(f(x,y)\big) \subseteq\text{Rad}\big(xy^q-x^qy-1\big)=\big(xy^q-x^qy-1\big),\]
where the last equality holds because $xy^q-x^qy-1$ is an irreducible polynomial in $\FF[x,y]$ by \cite[Proposition 2.1.1]{bonnafe}. Hence
$xy^q-x^qy-1$ divides $f(x,y) \text{ in } \FF[x,y]$, which can only happen if $f(x,y)=0$ in $\FF[x,y]$ because otherwise $f(x,y)$ is of finite total degree at most $q-2$. Therefore, the $\omega_{i,j}$ are linearly independent.
\end{proof}

\begin{rmk}
The differentials $\omega_{i,j}$ defined in \cref{basis} have the following vanishing order at $P=[x_0:y_0:z_0] \in C$.

\begin{empheq}[left=\empheqlbrace]{align*}
0 \ \ \ \ \ \ \ \ \ & \text{ if } z_0 \neq 0 \cr
q-2-(i+j) & \text{ if } z_0=0, \ x_0 \neq 0 \text{ and } y_0 \neq 0 \cr
q-2-i+jq & \text{ if } P=[1:0:0] \cr
q-2-j+iq & \text{ if } P=[0:1:0]
\end{empheq}

This follows from the proof above and the fact that
\[\ord_P(s)=\min\{\ord_P(s^q),\ord_P(s)\}=\ord_P(s^q-s)=\ord_P(t^{q+1})=q+1\]
if $P=[1:0:0]$ and similarly that $\ord_P(v) = q+1$ if $P=[0:1:0]$.
\end{rmk}

\section{The decomposition of \texorpdfstring{$M$}{M} as an \texorpdfstring{$\FF[G]$}{F[G]}-module}\label{section decomposition}
By the Krull-Schmidt theorem, $M$ has a decomposition as a direct sum of indecomposable right $\FF[G]$-modules which are unique up to isomorphism and permutation. The goal of this section is to find these indecomposable $\FF[G]$-modules and their multiplicity.

First, we compute the action of an element of $G$ on the basis of $M$ given in \cref{basis}.

\begin{lemma}\label{action on basis}
Let $0 \leq i,j \leq q-2$ such that $i+j \leq q-2$, and let ${g =\abcd \in G}$. Then
\[g^*(\omega_{i,j})=\frac{(\alpha x+\beta y)^i(\gamma x+\delta y)^j}{x^q}dx.\]
\end{lemma}

\begin{proof}
We have $g^*(x) = \alpha x+\beta y$ and $g^*(y) = \gamma x+\delta y$. Using the fact that $\frac{1}{x^q}dx=\frac{1}{y^q}dy$ as in the proof of \cref{basis}, this implies that
\[g^*(dx) = d(\alpha x+\beta y)=\left(\alpha+\beta \left(\frac{y}{x}\right)^q\right)dx=\frac{\alpha x^q+\beta y^q}{x^q}dx.\]
Putting these facts together, we obtain that
\begin{align*}
g^*( \omega_{i,j}) & = \frac{(\alpha x+\beta y)^i(\gamma x+\delta y)^j}{(\alpha x+\beta y)^q}\frac{\alpha x^q+\beta y^q}{x^q}dx \\
& = \frac{(\alpha x+\beta y)^i(\gamma x+\delta y)^j}{x^q}dx,
\end{align*}
as was to be shown.
\end{proof}

The next result states that the $\FF[G]$-module $M$ can be expressed as a direct sum of the $\FF[G]$-modules $V^k= \mathrm{Sym}^k(\FF_q^2)$, $k=0, \ldots, q-2$; these are obviously not isomorphic to each other and indecomposable as we've seen in the proof of \cref{decomposition thm}.

\begin{corollary}\label{decomposition}
The following isomorphism of $\FF[G]$-modules holds:
\[M \cong \bigoplus_{k=0}^{q-2} V^k\]
\end{corollary}

\begin{proof}
Let us define the degree of each $\omega_{i,j}$ as $i+j$. Then by \cref{action on basis}, each $\omega_{i,j}$ is mapped to a sum of basis elements of degree $i+j$ under the action of any element of $G$. Therefore, if $W^k$ is the subspace of $M$ with basis ${\{\omega_{i,j} \ | \ i+j=k\}}$ then $M \cong \oplus_{k=0}^{q-2} W^k$ as $\FF[G]$-modules, because each $W^k$ is stable under the action of $G$ and the $\omega_{i,j}$ form a basis of $M$ by \cref{basis}. Furthermore, by \cref{action on basis}, mapping $\omega_{i,j}$ to $(1,0)^i \cdot (0,1)^j$ in $\mathrm{Sym}^k(\FF_q^2) = V^k$ evidently defines an isomorphism of right $\FF[G]$-modules. Combining these two facts finishes the proof.
\end{proof}

\begin{rmk}\label{remark geometric interpretation}
Let $\mathbf{G}=SL_2\big(\FF\big)$ and consider the standard left action of $\mathbf{G}$ on~$\FF^2$. Then, the left $\FF[\mathbf{G}]$-modules $\Delta(k) := \mathrm{Sym}^k(\FF^2)$, $k \ge 0$, play a crucial role in \cite[Section 10]{bonnafe}. Furthermore, let the right $\FF[\mathbf{G}]$-module $\tilde{\Delta}(k)$ be obtained from the left $\FF[\mathbf{G}]$-module $\Delta(k)$ via transposition on $\mathbf{G}$. We then have ${\text{Res}^\mathbf{G}_G\tilde{\Delta}(k)\cong V^k}$. In particular, \cref{decomposition} gives a geometric realisation of the modules $\mathrm{Res}^{\mathbf{G}}_G\tilde{\Delta}(k)$, thus extending the geometric viewpoint omnipresent in \cite{bonnafe}.
\end{rmk}

{\footnotesize

\bibliography{bibliography}
\bibliographystyle{alpha}
}

\end{document}